\documentclass[11pt,reqno]{amsart} % please use amsart at 11pt

\usepackage{amssymb,latexsym}
\usepackage{cite} % to get Refs. [1,2,3] typeset as [1--3] automatically

\usepackage[height=190mm,width=130mm]{geometry} % this is the journal
\usepackage{mathrsfs,upref}         % not so essential, but part of journal style
\usepackage{mathptmx}		    % Journal is printed with poscript fonts:
\theoremstyle{plain}
\newtheorem{theorem}{Theorem}

\newtheorem{corollary}{Corollary}

\theoremstyle{definition}
\newtheorem{definition}{Definition}

\theoremstyle{remark}
\newtheorem{remark}{Remark}

\numberwithin{equation}{section} % to get equations numbered
\makeatletter
\newcommand{\vast}{\bBigg@{1}}
\newcommand{\Vast}{\bBigg@{2}}
\makeatother
\begin{document}
\title[Affine and Functional Form of Jensen's Inequalitiy]{Affine and Functional Form of Jensen's Inequalitiy
for $3$-convex Functions at a Point} % please provide
                                % an abbreviated title

\author{Imran Abbas Baloch}
\address{GC University\\ Abdus Salam School of Mathematical Sciences,\\
  Lahore,\\ Pakistan.}
\email{iabbasbaloch@gmail.com}

\thanks{The first author was supported in part by the Higher Education Commission of Pakistan.}

\author{Silvestru Sever Dragomir}

\address{Mathematics, College of Engineering and Science\\ Victoria University\\
  Melbourne City,\\ Australia.}

%\curraddr{Other Institution\\ Some Other Department\\ P.~O. Box 0002 \\ 12347
  %Melbourne City,\\ Other Country}

\email{sever.dragomir@vu.edu.au}

\begin{abstract}
 In this paper, we give the refinement of an extension of Jensen's inequality
to affine combinations. Furthermore, we present the functional form of
Jensen's inequality for continuous 3-convex functions of one variable at a
point.
\end{abstract}

%\dedicatory{This paper is dedicated to Professor X on his 125th birthday.}



\maketitle

\section{Introduction}

Let $\mathcal{X}$ be a real linear space. A set $\mathcal{A} \subset \mathcal{X}$ is affine if it contain all binomial affine combinations $\alpha a + \beta b$ of points $a, b \in \mathcal{A}$ and coefficient $ \alpha , \beta \in \mathbb{R}$ of sum $\alpha  + \beta = 1$. The affine hull of a set $\mathcal{I} \subset \mathcal{X}$ as the smallest affine set that contains $\mathcal{I}$ is denoted with aff$\mathcal{I}$. A function $f : \mathcal{A} \rightarrow \mathbb{R}$ is affine if the equality
$$ f(\alpha a + \beta b) = \alpha f(a) + \beta f(b)     $$
holds for all binomial affine combinations of points of $\mathcal{A}$.\\
A set $\mathcal{Y} \subset \mathcal{X}$ is convex if it contains all binomial convex combinations $\alpha a + \beta b$ of points $a, b \in \mathcal{Y}$ and non-negative coefficient $ \alpha , \beta \in \mathbb{R}$ of sum $\alpha  + \beta = 1$. The convex hull of of a set $\mathcal{I} \subset \mathcal{X}$ as the smallest convex set that contains $\mathcal{I}$ is denoted with conv$\mathcal{I}$. A function $f : \mathcal{Y} \rightarrow \mathbb{R}$ is convex if the inequality
$$ f(\alpha a + \beta b) = \alpha f(a) + \beta f(b)     $$
holds for all binomial convex combinations of points of $\mathcal{Y}$.\\
Let $\Omega$ be a non-empty set, and let $\mathbb{X}$ be a subspace of linear space of all real functions on the domain $\Omega$. Also, assume that the unit function defined by $I(x) =1 $ for every $x \in \Omega$ belongs to $\mathbb{X}$. Let $\mathcal{I} \subset \mathbb {R}$ be an interval, and let $\mathbb{X}_{\mathcal{I}} \subset \mathbb{X}$ be asubset containing all functions with image in $\mathcal{I}$. If $\alpha g + \beta h$ is a convex combination of functions $g, h \in \mathbb{X}_{\mathcal{I}}$, then the convex combination $\alpha g(x) + \beta h(x)$ is in $\mathcal{I}$ for every $x \in \Omega,$ which indicates that functions set $\mathbb{X}_{\mathcal{I}}$ is convex.\\
A linear functional $L: \mathbb{X} \rightarrow \mathbb{R}$ is positive (non-negative) if $L(g) \geq 0$ for every non-negative function $g \in \mathbb{X}$, and $L$ is unital(normalized) if $L(1) = 1$. If $ g \in \mathbb{X}$, then every unital positive functional $L$, the number $L(g)$ is in the closed interval of real numbers containing the image of the function $g$.\\
In 2015, Z. Pavi$\acute{c}$ \cite{ZP} gave the extension of Jensen's inequality to affine combinations in the following form
\begin{theorem}\label{IT1}
Let $\alpha_{i}, \beta_{j}, \gamma_{k} \geq 0$ be coefficients such that their sum $ \alpha = \sum_{i = 1}^{n}\alpha_{i}$, $\beta = \sum_{j = 1}^{m}\beta_{j}$, $\gamma = \sum_{k = 1}^{l}\gamma_{k}$ satisfy $ \alpha + \beta - \gamma = 1$ and $\alpha, \beta \in (0, 1]$. Let $a_{i}, b_{j}, c_{k} \in \mathbb{R}$ be points such that $c_{k} \in conv\{a, b\}$, where
\begin{equation}\label{IE1}
 a = \frac{1}{\alpha} \sum_{i = 1}^{n}\alpha_{i}a_{i}\;,\; b = \frac{1}{\beta} \sum_{j = 1}^{m}\beta_{j}b_{j} .
 \end{equation}
Then the affine combination
\begin{equation}\label{IE2}
  \sum_{i = 1}^{n}\alpha_{i}a_{i} +  \sum_{j = 1}^{m}\beta_{j}b_{j} - \sum_{k = 1}^{l}\gamma_{k}c_{k}
  \end{equation}
belongs to conv\{a, b\}, and for every convex functions $ f: conv\{a_{i}, b_{j}\} \rightarrow \mathbb{R}$ satisfies the inequality
\begin{equation}\label{IE3}
f \vast( \sum_{i = 1}^{n}\alpha_{i}a_{i} +  \sum_{j = 1}^{m}\beta_{j}b_{j} - \sum_{k = 1}^{l}\gamma_{k}c_{k} \vast) \leq \sum_{i = 1}^{n}\alpha_{i}f(a_{i}) +  \sum_{j = 1}^{m}\beta_{j}f(b_{j}) - \sum_{k = 1}^{l}\gamma_{k}f(c_{k}).
\end{equation}
\end{theorem}
In 2014, Z. Pavi$\acute{c}$ \cite{ZP1} also gave the functional form of jensen's inequality for the continuous convex functions of one variable in the following form
\begin{theorem}\label{IT2}
Let $\mathcal{I} \subset \mathbb{R}$ be a closed interval, let $[a, b] \subset \mathcal{I}$, let function $g \in \mathbb{X}_{[a, b]}$ and function $h \in \mathbb{X}_{\mathcal{I} \backslash (a, b)}$. Let $f: \mathcal{I} \rightarrow \mathbb{R}$ be a continuous convex function such that $f(g), f(h) \in \mathbb{X}$. If a pair of unital positive linear functionals $L, H :\mathbb{X} \rightarrow \mathbb{R}$ satisfies
\begin{equation}\label{IE4}
L(g) = H(h),
\end{equation}
then
\begin{equation}\label{IE5}
L(f(g)) \leq  H(f(h)).
\end{equation}
\end{theorem}
Furthermore, Z. Pavi$\acute{c}$ \cite{ZP1} also gave some consequent results in the form of corollaries and using these corollaries, he gave another important result as follow
\begin{corollary}\label{IC1}
Let $\mathcal{I} \subset \mathbb{R}$ be a closed interval, let function $g \in \mathbb{X}_{[a, b]}$. Let $f: \mathcal{I} \rightarrow \mathbb{R}$ be a continuous convex function such that $f(g) \in \mathbb{X}$. If a unital positive linear functional $L:\mathbb{X} \rightarrow \mathbb{R}$ satisfies the implication (\ref{IE4}) $\Rightarrow$ (\ref{IE5}) of Theorem \ref{IT2} for $L = H$, then
\begin{equation}\label{IE6}
f(L(g)) \leq L(f(g)).
\end{equation}
\end{corollary}
\begin{corollary}\label{IC2}
Let $[a_{1}, b_{1}] \subseteq...\subseteq [a_{n - 1}, b_{n - 1}] \subseteq \mathcal{I}.$ Let function $g_{1} \in \mathbb{X}_{[a_{1}, b_{1}]}$, let functions $g_{k} \in \mathbb{X}_{[a_{k}, b_{k}]\backslash (a_{k - 1}, b_{k - 1})}$ for $k = 2,...,n - 1$, and let function $g_{n} \in \mathbb{X}_{\mathcal{I} \backslash (a_{n -1}, b_{n -1})}.$ Let $f: \mathcal{I} \rightarrow \mathbb{R}$ be a continuous convex function such that $f(g_{i}) \in \mathbb{X}$. \\
If an n-tuple of unital positive linear functionals $L_{i}:\mathbb{X} \rightarrow \mathbb{R}$ satisfies
\begin{equation}\label{IE7}
L_{i}(g_{i}) =  L_{i + 1}(g_{i + 1})\;\; for\; i =1,...,n-1,
\end{equation}
then
\begin{equation}\label{IE8}
L_{i}(f(g_{i})) \leq  L_{i + 1}((g_{i + 1}))\;\; for\; i =1,...,n-1.
\end{equation}
\end{corollary}
\begin{corollary}\label{IC3}
Let $\mathcal{I} \subset \mathbb{R}$ be a closed interval, and let functions $g_{1},...,g_{n} \in \mathbb{X}_{\mathcal{I}}$. Let $f: \mathcal{I} \rightarrow \mathbb{R}$ be a continuous convex function such that $f(g_{i}) \in \mathbb{X}$.\\
Then every n-tuple of positive linear functionals $L_{i}:\mathbb{X} \rightarrow \mathbb{R}$
with $\sum_{i =1}^{n} L_{i}(1) = 1$ satisfies the inclusion
\begin{equation}\label{IE9}
\sum_{i =1}^{n}L_{i}(g_{i}) \in \mathcal{I}
\end{equation}
and the inequality
\begin{equation}\label{IE10}
f\vast(\sum_{i =1}^{n}L_{i}(g_{i})\vast) \leq \sum_{i =1}^{n}L_{i}(f(g_{i})).
\end{equation}
\end{corollary}
\begin{theorem}\label{IT3}
Let $\mathcal{I} \subset \mathbb{R}$ be a closed interval, and let $[a, b] \subset \mathcal{I}$. Let functions $g_{1},...,g_{n} \in \mathbb{X}_{[a, b]}$ and $h_{1},...,h_{m} \in \mathbb{X}_{\mathcal{I}\backslash (a, b)}$. Let $f: \mathcal{I} \rightarrow \mathbb{R}$ be a continuous convex function such that $f(g_{i}), f(h_{j}) \in \mathbb{X}$.\\
If a pair of n-tuple of positive linear functionals $L_{i}, H_{j}:\mathbb{X} \rightarrow \mathbb{R}$ with $\sum_{i = 1}^{n} L_{i}(1) = \sum_{j = 1}^{m} H_{j}(1) = 1$ satisfies
\begin{equation}\label{IE11}
\sum_{i = 1}^{n} L_{i}(g_{i}) = \sum_{j = 1}^{m} H_{j}(h_{j}),
\end{equation}
then
\begin{equation}\label{IE12}
\sum_{i = 1}^{n} L_{i}(f(g_{i})) \leq \sum_{j = 1}^{m} H_{j}(f(h_{j})).
\end{equation}
\end{theorem}

%\verb1\begin1\verb1{align} asddd  1\verb1\end1\verb1{align}1\\
%\begin{center}
%  \verb1\begin1\verb1{align} ...  1\verb1\end1\verb1{align}1\\
%  \verb1\begin1\verb1{align*} ...  1\verb1\end1\verb1{align*}1\\
%  \verb1\begin1\verb1{gather} ...  1\verb1\end1\verb1{gather}1\\
%  \verb1\begin1\verb1{gather*} ...  1\verb1\end1\verb1{gather*}1\\
%  \verb1\begin1\verb1{multline} ...  1\verb1\end1\verb1{multline}1\\
%  \verb1\begin1\verb1{multline*} ...  1\verb1\end1\verb1{multline*}1\\
%  \verb1\begin1\verb1{equation}1\verb1\begin1\verb1{split} ...1 \verb2\begin2\verb2{equation}2\verb2\begin2\verb2{split}2
%\end{center}
%and the number them by using the corresponding
% \verb1\label1\verb1{...}1 and  \verb1\tag1\verb1{...}1 elements.
%
%Please do not insert any manual skips, just rely upon the style.

\section{Results}
In \cite{BPP}, I. A. Baloch, J. Pe\v{c}ari\v{c}, M. Praljak defined a new class of functions  which is defined as follow
\begin{definition}\label{MD1}
Let $c \in I^{\circ}$, where $I$ is an arbitrary interval(open, closed or semi-open in either direction) in $\mathbb{R}$ and $I^{\circ}$ is its interior. We say that $f:I \rightarrow \mathbb{R}$ is $3$-convex function in point $c$ (respectively $3$-concave function in point $c$) if there exists a constant $A$ such that the function $F(x) = f(x) - \frac{A}{2} x^{2}$ is  concave (resp. convex) on $I \cap (-\infty,c]$ and convex (resp. concave) on $I \cap [c,\infty)$. A $f$ is $3$-concave function in point $c$ if $-f$ is $3$-convex function in point $c$.
\end{definition}
A property that explains the name of the class is the fact that a function is $3$-convex on an interval if and only if it is $3$-convex at every point of the interval (see \cite{BPP}). Note that $K^{c}_{1}(I)$ and $K^{c}_{2}(I)$) denote the class of all $3$-convex functions in point $c$ and the class of all $3$-concave functions in point $c$ respectively.
\begin{theorem}\label{MT1}
Let $\alpha_{i}, \beta_{j}, \gamma_{k} \geq 0$ and $\lambda_{i}, \mu_{j}, \nu_{k} \geq 0$  be coefficients such that their sum $ \alpha = \sum_{i = 1}^{n}\alpha_{i}$, $\beta = \sum_{j = 1}^{m}\beta_{j}$, $\gamma = \sum_{k = 1}^{l}\gamma_{k}$ satisfy $ \alpha + \beta - \gamma = 1$ and $\alpha, \beta \in (0, 1]$; $ \lambda = \sum_{i = 1}^{n}\lambda_{i}$, $\mu = \sum_{j = 1}^{m}\mu_{j}$, $\nu = \sum_{k = 1}^{l}\nu_{k}$ satisfy $ \lambda + \mu - \nu = 1$ and $\lambda, \mu \in (0, 1].$ Let $a_{i}, b_{j}, c_{k} \in [a, c]$ be points such that $c_{k} \in conv\{a_{i}, b_{j}\}$ and $r_{i}, s_{j}, t_{k} \in [c, b]$ be points such that $t_{k} \in conv\{r_{i}, s_{j}\}$ , where
$$   a = \frac{1}{\alpha} \sum_{i = 1}^{n}\alpha_{i}a_{i}\;,\; b = \frac{1}{\beta} \sum_{j = 1}^{m}\beta_{j}b_{j}\;,\;  r = \frac{1}{\lambda} \sum_{i = 1}^{n}\lambda_{i}r_{i}\;,\; s = \frac{1}{\mu} \sum_{j = 1}^{m}\mu_{j}s_{j} .        $$
Now, if
$$ \sum_{i = 1}^{n}\alpha_{i}(a_{i})^{2} + \sum_{j = 1}^{m}\beta_{j}(b_{j})^{2} - \sum_{k = 1}^{l}\gamma_{k}(c_{k})^{2} - \vast(\sum_{i = 1}^{n}\alpha_{i}a_{i} + \sum_{j = 1}^{m}\beta_{j}b_{j} - \sum_{k = 1}^{l}\gamma_{k}c_{k}\vast)^{2}$$
 \begin{equation}\label{ME1}
 = \sum_{i = 1}^{n}\lambda_{i}(r_{i})^{2} + \sum_{j = 1}^{m}\mu_{j}(s_{j})^{2} - \sum_{k = 1}^{l}\nu_{k}(t_{k})^{2} - \vast(\sum_{i = 1}^{n}\lambda_{i}r_{i} + \sum_{j = 1}^{m}\mu_{j}s_{j} - \sum_{k = 1}^{l}\nu_{k}t_{k}\vast )^{2}
 \end{equation}
 and also there exists $c \in I^{\circ}$ ($I = [a, b]$) such that
 \begin{equation}\label{ME2}
 \max\{ \max_{i}\{a_{i}\}, \max_{j}\{b_{j}\}, \max_{k}\{c_{k}\} \}\; \leq\; c \;\leq\; \min\{ \min_{i}\{r_{i}\}, \min_{j}\{s_{j}\}, \min_{k}\{t_{k}\} \}.
 \end{equation}
 Then for every $f \in K^{c}_{1}(I)$, the following inequality holds
 $$\sum_{i = 1}^{n}\alpha_{i}f(a_{i}) + \sum_{j = 1}^{m}\beta_{j}f(b_{j}) - \sum_{k = 1}^{l}\gamma_{k}f(c_{k}) - f\vast(\sum_{i = 1}^{n}\alpha_{i}a_{i} + \sum_{j = 1}^{m}\beta_{j}b_{j} - \sum_{k = 1}^{l}\gamma_{k}c_{k}\vast )$$
 \begin{equation}\label{ME3}
 \leq \sum_{i = 1}^{n}\alpha_{i}f(r_{i}) + \sum_{j = 1}^{m}\beta_{j}f(s_{j}) - \sum_{k = 1}^{l}\gamma_{k}f(t_{k}) - f\vast(\sum_{i = 1}^{n}\alpha_{i}r_{i} + \sum_{j = 1}^{m}\beta_{j}s_{j} - \sum_{k = 1}^{l}\gamma_{k}t_{k}\vast )
 \end{equation}
\end{theorem}
\begin{proof}
Since $f \in K^{c}_{1}(I)$, then there exists a constant $A$ such that $F(x) = f(x) - \frac{A}{2} x^{2}$ is concave on $I \cap (-\infty,c]$ and for $a_{i}, b_{j}, c_{k} \in [a, c]$ be points such that $c_{k} \in conv\{a_{i}, b_{j}\}$, so by using inequality (\ref{IE3}) we have
\begin{eqnarray*}
 0 &\geq& \sum_{i = 1}^{n}\alpha_{i}F(a_{i}) + \sum_{j = 1}^{m}\beta_{j}F(b_{j}) - \sum_{k = 1}^{l}\gamma_{k}F(c_{k}) - F\vast(\sum_{i = 1}^{n}\alpha_{i}a_{i} + \sum_{j = 1}^{m}\beta_{j}b_{j} - \sum_{k = 1}^{l}\gamma_{k}c_{k}\vast )\\
 &=&\sum_{i = 1}^{n}\alpha_{i}f(a_{i}) + \sum_{j = 1}^{m}\beta_{j}f(b_{j}) - \sum_{k = 1}^{l}\gamma_{k}f(c_{k}) - f\vast(\sum_{i = 1}^{n}\alpha_{i}a_{i} + \sum_{j = 1}^{m}\beta_{j}b_{j} - \sum_{k = 1}^{l}\gamma_{k}c_{k}\vast )\\
 &-& \frac{A}{2}\vast \{ \sum_{i = 1}^{n}\alpha_{i}(a_{i})^{2} + \sum_{j = 1}^{m}\beta_{j}(b_{j})^{2} - \sum_{k = 1}^{l}\gamma_{k}(c_{k})^{2} - \vast(\sum_{i = 1}^{n}\alpha_{i}a_{i} + \sum_{j = 1}^{m}\beta_{j}b_{j} - \sum_{k = 1}^{l}\gamma_{k}c_{k}\vast )^{2} \vast\}
 \end{eqnarray*}
 Also, since $f \in K^{c}_{1}(I)$ is convex on $I \cap [c,\infty)$, hence for $r_{i}, s_{j}, t_{k} \in [c, b]$ be points such that $t_{k} \in conv\{r_{i}, s_{j}\}$, so by using inequality (\ref{IE3}) we have
 \begin{eqnarray*}
 0 &\leq& \sum_{i = 1}^{n}\alpha_{i}F(r_{i}) + \sum_{j = 1}^{m}\beta_{j}F(s_{j}) - \sum_{k = 1}^{l}\gamma_{k}F(t_{k}) - F\vast(\sum_{i = 1}^{n}\alpha_{i}r_{i} + \sum_{j = 1}^{m}\beta_{j}s_{j} - \sum_{k = 1}^{l}\gamma_{k}t_{k}\vast )\\
 &=&\sum_{i = 1}^{n}\alpha_{i}f(r_{i}) + \sum_{j = 1}^{m}\beta_{j}f(s_{j}) - \sum_{k = 1}^{l}\gamma_{k}f(t_{k}) - f\vast(\sum_{i = 1}^{n}\alpha_{i}r_{i} + \sum_{j = 1}^{m}\beta_{j}s_{j} - \sum_{k = 1}^{l}\gamma_{k}t_{k}\vast )\\
 &-& \frac{A}{2}\vast \{ \sum_{i = 1}^{n}\alpha_{i}(r_{i})^{2} + \sum_{j = 1}^{m}\beta_{j}(s_{j})^{2} - \sum_{k = 1}^{l}\gamma_{k}(t_{k})^{2} - \vast(\sum_{i = 1}^{n}\alpha_{i}r_{i} + \sum_{j = 1}^{m}\beta_{j}s_{j} - \sum_{k = 1}^{l}\gamma_{k}t_{k}\vast )^{2} \vast\}
 \end{eqnarray*}
 From above, we have
 $$\sum_{i = 1}^{n}\alpha_{i}f(a_{i}) + \sum_{j = 1}^{m}\beta_{j}f(b_{j}) - \sum_{k = 1}^{l}\gamma_{k}f(c_{k}) - f\vast(\sum_{i = 1}^{n}\alpha_{i}a_{i} + \sum_{j = 1}^{m}\beta_{j}b_{j} - \sum_{k = 1}^{l}\gamma_{k}c_{k}\vast )$$
$$ - \frac{A}{2}\vast \{ \sum_{i = 1}^{n}\alpha_{i}(a_{i})^{2} + \sum_{j = 1}^{m}\beta_{j}(b_{j})^{2} - \sum_{k = 1}^{l}\gamma_{k}(c_{k})^{2} - \vast(\sum_{i = 1}^{n}\alpha_{i}a_{i} + \sum_{j = 1}^{m}\beta_{j}b_{j} - \sum_{k = 1}^{l}\gamma_{k}c_{k}\vast )^{2} \vast\}$$
$$ \leq 0 \leq $$
$$\sum_{i = 1}^{n}\alpha_{i}f(r_{i}) + \sum_{j = 1}^{m}\beta_{j}f(s_{j}) - \sum_{k = 1}^{l}\gamma_{k}f(t_{k}) - f\vast(\sum_{i = 1}^{n}\alpha_{i}r_{i} + \sum_{j = 1}^{m}\beta_{j}s_{j} - \sum_{k = 1}^{l}\gamma_{k}t_{k}\vast )$$
 \begin{equation}\label{ME4}
 - \frac{A}{2}\vast \{ \sum_{i = 1}^{n}\alpha_{i}(r_{i})^{2} + \sum_{j = 1}^{m}\beta_{j}(s_{j})^{2} - \sum_{k = 1}^{l}\gamma_{k}(t_{k})^{2} - \vast(\sum_{i = 1}^{n}\alpha_{i}r_{i} + \sum_{j = 1}^{m}\beta_{j}s_{j} - \sum_{k = 1}^{l}\gamma_{k}t_{k}\vast )^{2} \vast\}
 \end{equation}
 So
$$\sum_{i = 1}^{n}\alpha_{i}f(a_{i}) + \sum_{j = 1}^{m}\beta_{j}f(b_{j}) - \sum_{k = 1}^{l}\gamma_{k}f(c_{k}) - f\vast(\sum_{i = 1}^{n}\alpha_{i}a_{i} + \sum_{j = 1}^{m}\beta_{j}b_{j} - \sum_{k = 1}^{l}\gamma_{k}c_{k}\vast )$$
$$ - \frac{A}{2}\vast \{ \sum_{i = 1}^{n}\alpha_{i}(a_{i})^{2} + \sum_{j = 1}^{m}\beta_{j}(b_{j})^{2} - \sum_{k = 1}^{l}\gamma_{k}(c_{k})^{2} - \vast(\sum_{i = 1}^{n}\alpha_{i}a_{i} + \sum_{j = 1}^{m}\beta_{j}b_{j} - \sum_{k = 1}^{l}\gamma_{k}c_{k}\vast )^{2} \vast\}$$
$$ \leq \sum_{i = 1}^{n}\alpha_{i}f(r_{i}) + \sum_{j = 1}^{m}\beta_{j}f(s_{j}) - \sum_{k = 1}^{l}\gamma_{k}f(t_{k}) - f\vast(\sum_{i = 1}^{n}\alpha_{i}r_{i} + \sum_{j = 1}^{m}\beta_{j}s_{j} - \sum_{k = 1}^{l}\gamma_{k}t_{k}\vast )$$
 $$- \frac{A}{2}\vast \{ \sum_{i = 1}^{n}\alpha_{i}(r_{i})^{2} + \sum_{j = 1}^{m}\beta_{j}(s_{j})^{2} - \sum_{k = 1}^{l}\gamma_{k}(t_{k})^{2} - \vast(\sum_{i = 1}^{n}\alpha_{i}r_{i} + \sum_{j = 1}^{m}\beta_{j}s_{j} - \sum_{k = 1}^{l}\gamma_{k}t_{k}\vast )^{2} \vast\}$$
 by using (\ref{ME1}), we get (\ref{ME3}).
\end{proof}
\begin{remark}\label{MR1}
From the proof of Theorem \ref{MT1}, we have
$$\sum_{i = 1}^{n}\alpha_{i}f(a_{i}) + \sum_{j = 1}^{m}\beta_{j}f(b_{j}) - \sum_{k = 1}^{l}\gamma_{k}f(c_{k}) - f\vast(\sum_{i = 1}^{n}\alpha_{i}a_{i} + \sum_{j = 1}^{m}\beta_{j}b_{j} - \sum_{k = 1}^{l}\gamma_{k}c_{k}\vast )$$
\begin{equation}\label{ME5}
\leq \frac{A}{2}\vast \{ \sum_{i = 1}^{n}\alpha_{i}(a_{i})^{2} + \sum_{j = 1}^{m}\beta_{j}(b_{j})^{2} - \sum_{k = 1}^{l}\gamma_{k}(c_{k})^{2} - \vast(\sum_{i = 1}^{n}\alpha_{i}a_{i} + \sum_{j = 1}^{m}\beta_{j}b_{j} - \sum_{k = 1}^{l}\gamma_{k}c_{k}\vast )^{2} \vast\}
\end{equation}
and
$$\sum_{i = 1}^{n}\alpha_{i}f(r_{i}) + \sum_{j = 1}^{m}\beta_{j}f(s_{j}) - \sum_{k = 1}^{l}\gamma_{k}f(t_{k}) - f\vast(\sum_{i = 1}^{n}\alpha_{i}r_{i} + \sum_{j = 1}^{m}\beta_{j}s_{j} - \sum_{k = 1}^{l}\gamma_{k}t_{k}\vast )$$
\begin{equation}\label{ME6}
\geq \frac{A}{2}\vast \{ \sum_{i = 1}^{n}\alpha_{i}(r_{i})^{2} + \sum_{j = 1}^{m}\beta_{j}(s_{j})^{2} - \sum_{k = 1}^{l}\gamma_{k}(t_{k})^{2} - \vast(\sum_{i = 1}^{n}\alpha_{i}r_{i} + \sum_{j = 1}^{m}\beta_{j}s_{j} - \sum_{k = 1}^{l}\gamma_{k}t_{k}\vast )^{2} \vast\}
\end{equation}
So under assumption (\ref{ME1}), we can get a improvement of (\ref{ME3}) as follows
$$\sum_{i = 1}^{n}\alpha_{i}f(a_{i}) + \sum_{j = 1}^{m}\beta_{j}f(b_{j}) - \sum_{k = 1}^{l}\gamma_{k}f(c_{k}) - f\vast(\sum_{i = 1}^{n}\alpha_{i}a_{i} + \sum_{j = 1}^{m}\beta_{j}b_{j} - \sum_{k = 1}^{l}\gamma_{k}c_{k}\vast )$$
$$ \leq \frac{A}{2}\vast \{ \sum_{i = 1}^{n}\alpha_{i}(a_{i})^{2} + \sum_{j = 1}^{m}\beta_{j}(b_{j})^{2} - \sum_{k = 1}^{l}\gamma_{k}(c_{k})^{2} - \vast(\sum_{i = 1}^{n}\alpha_{i}a_{i} + \sum_{j = 1}^{m}\beta_{j}b_{j} - \sum_{k = 1}^{l}\gamma_{k}c_{k}\vast )^{2} \vast\}  $$
$$\vast( = \frac{A}{2}\vast \{ \sum_{i = 1}^{n}\alpha_{i}(r_{i})^{2} + \sum_{j = 1}^{m}\beta_{j}(s_{j})^{2} - \sum_{k = 1}^{l}\gamma_{k}(t_{k})^{2} - \vast(\sum_{i = 1}^{n}\alpha_{i}r_{i} + \sum_{j = 1}^{m}\beta_{j}s_{j} - \sum_{k = 1}^{l}\gamma_{k}t_{k}\vast )^{2} \vast\} \vast)$$
\begin{equation}\label{ME7}
\leq \sum_{i = 1}^{n}\alpha_{i}f(r_{i}) + \sum_{j = 1}^{m}\beta_{j}f(s_{j}) - \sum_{k = 1}^{l}\gamma_{k}f(t_{k}) - f\vast(\sum_{i = 1}^{n}\alpha_{i}r_{i} + \sum_{j = 1}^{m}\beta_{j}s_{j} - \sum_{k = 1}^{l}\gamma_{k}t_{k}\vast )
\end{equation}
\end{remark}
Assume that $\tilde{a} =\max_{i}\{a_{i}\}, \tilde{b} = \max_{j}\{b_{j}\}, \tilde{c} = \max_{k}\{c_{k}\}$ and $\tilde{r} = \min_{i}\{r_{i}\}, \tilde{s} =  \min_{j}\{s_{j}\},  \tilde{t} = \min_{k}\{t_{k}\}$. Also, let $\tilde{\tilde{a}}= \max \{\tilde{a}, \tilde{b}, \tilde{c}\}$ and $\tilde{\tilde{r}}= \min \{\tilde{r}, \tilde{s}, \tilde{t}\}$
Now, we give the next result which weakens the assumption (\ref{ME1}) such that inequality (\ref{ME5}) also holds.
\begin{theorem}\label{MT2}
Let $\alpha_{i}, \beta_{j}, \gamma_{k} \geq 0$ and $\lambda_{i}, \mu_{j}, \nu_{k} \geq 0$  be coefficients such that their sum $ \alpha = \sum_{i = 1}^{n}\alpha_{i}$, $\beta = \sum_{j = 1}^{m}\beta_{j}$, $\gamma = \sum_{k = 1}^{l}\gamma_{k}$ satisfy $ \alpha + \beta - \gamma = 1$ and $\alpha, \beta \in (0, 1]$; $ \lambda = \sum_{i = 1}^{n}\lambda_{i}$, $\mu = \sum_{j = 1}^{m}\mu_{j}$, $\nu = \sum_{k = 1}^{l}\nu_{k}$ satisfy $ \lambda + \mu - \nu = 1$ and $\lambda, \mu \in (0, 1].$ Let $a_{i}, b_{j}, c_{k} \in [a, c]$ be points such that $c_{k} \in conv\{a_{i}, b_{j}\}$ and $r_{i}, s_{j}, t_{k} \in [c, b]$ be points such that $t_{k} \in conv\{r_{i}, s_{j}\}$ , where
$$   a = \frac{1}{\alpha} \sum_{i = 1}^{n}\alpha_{i}a_{i}\;,\; b = \frac{1}{\beta} \sum_{j = 1}^{m}\beta_{j}b_{j}\;,\;  r = \frac{1}{\lambda} \sum_{i = 1}^{n}\lambda_{i}r_{i}\;,\; s = \frac{1}{\mu} \sum_{j = 1}^{m}\mu_{j}s_{j} . $$ such that
\begin{equation}\label{ME8}
 \tilde{\tilde{a}} \leq \tilde{\tilde{r}}
 \end{equation}
 and $f \in K^{c}_{1}(I)$ for some $c \in [\tilde{\tilde{a}}, \tilde{\tilde{r}}]$. Then if\\
    (a)\\
    $$f''_{-}(\tilde{\tilde{a}}) \geq 0$$
    and
    $$ \sum_{i = 1}^{n}\alpha_{i}(a_{i})^{2} + \sum_{j = 1}^{m}\beta_{j}(b_{j})^{2} - \sum_{k = 1}^{l}\gamma_{k}(c_{k})^{2} - \vast(\sum_{i = 1}^{n}\alpha_{i}a_{i} + \sum_{j = 1}^{m}\beta_{j}b_{j} - \sum_{k = 1}^{l}\gamma_{k}c_{k}\vast )^{2}$$
    $$
    \leq \sum_{i = 1}^{n}\lambda_{i}(r_{i})^{2} + \sum_{j = 1}^{m}\mu_{j}(s_{j})^{2} - \sum_{k = 1}^{l}\nu_{k}(t_{k})^{2} - \vast(\sum_{i = 1}^{n}\lambda_{i}r_{i} + \sum_{j = 1}^{m}\mu_{j}s_{j} - \sum_{k = 1}^{l}\nu_{k}t_{k}\vast )^{2}
    $$
    or\\
    (b)\\
    $$f''_{+}(\tilde{\tilde{r}}) \leq 0$$
    and
    $$ \sum_{i = 1}^{n}\alpha_{i}(a_{i})^{2} + \sum_{j = 1}^{m}\beta_{j}(b_{j})^{2} - \sum_{k = 1}^{l}\gamma_{k}(c_{k})^{2} - \vast(\sum_{i = 1}^{n}\alpha_{i}a_{i} + \sum_{j = 1}^{m}\beta_{j}b_{j} - \sum_{k = 1}^{l}\gamma_{k}c_{k}\vast )^{2}$$
    $$
    \geq \sum_{i = 1}^{n}\lambda_{i}(r_{i})^{2} + \sum_{j = 1}^{m}\mu_{j}(s_{j})^{2} - \sum_{k = 1}^{l}\nu_{k}(t_{k})^{2} - \vast(\sum_{i = 1}^{n}\lambda_{i}r_{i} + \sum_{j = 1}^{m}\mu_{j}s_{j} - \sum_{k = 1}^{l}\nu_{k}t_{k}\vast )^{2}
    $$
    or\\
    (c)
    $$  f''_{-}(\tilde{\tilde{a}}) < 0 < f''_{+}(\tilde{\tilde{r}}) \; \;and \;\; f\; is \; 3-convex,$$
    then (\ref{ME3}) holds.
\end{theorem}
\begin{proof}
The idea of proof is similar to proof of Theorem \ref{MT1}. Hence, by proceeding as in the proof of Theorem \ref{MT1}. From the inequality \ref{ME4}, we have
$$\frac{A}{2}\vast[ \sum_{i = 1}^{n}\alpha_{i}(r_{i})^{2} + \sum_{j = 1}^{m}\beta_{j}(s_{j})^{2} - \sum_{k = 1}^{l}\gamma_{k}(t_{k})^{2} - \vast(\sum_{i = 1}^{n}\alpha_{i}r_{i} + \sum_{j = 1}^{m}\beta_{j}s_{j} - \sum_{k = 1}^{l}\gamma_{k}t_{k}\vast )^{2} $$
$$- \vast \{ \sum_{i = 1}^{n}\alpha_{i}(a_{i})^{2} + \sum_{j = 1}^{m}\beta_{j}(b_{j})^{2} - \sum_{k = 1}^{l}\gamma_{k}(c_{k})^{2} - \vast(\sum_{i = 1}^{n}\alpha_{i}a_{i} + \sum_{j = 1}^{m}\beta_{j}b_{j} - \sum_{k = 1}^{l}\gamma_{k}c_{k}\vast )^{2} \vast\}\vast]$$
$$\leq \sum_{i = 1}^{n}\alpha_{i}f(r_{i}) + \sum_{j = 1}^{m}\beta_{j}f(s_{j}) - \sum_{k = 1}^{l}\gamma_{k}f(t_{k}) - f\vast(\sum_{i = 1}^{n}\alpha_{i}r_{i} + \sum_{j = 1}^{m}\beta_{j}s_{j} - \sum_{k = 1}^{l}\gamma_{k}t_{k}\vast )$$
$$
- \vast\{\sum_{i = 1}^{n}\alpha_{i}f(a_{i}) + \sum_{j = 1}^{m}\beta_{j}f(b_{j}) - \sum_{k = 1}^{l}\gamma_{k}f(c_{k}) - f\vast(\sum_{i = 1}^{n}\alpha_{i}a_{i} + \sum_{j = 1}^{m}\beta_{j}b_{j} - \sum_{k = 1}^{l}\gamma_{k}c_{k}\vast )\vast\}
$$
Now, due to the concavity of $F$ on $[a, c]$ and convexity of $F$ on $[c, b]$, for every distinct points $a_{j} \in [a, \tilde{\tilde{a}}]$ and $r_{j} \in [\tilde{\tilde{r}}, b]$, $j = 1, 2, 3,$ we have
$$     [a_{1}, a_{2}, a_{3}]f \leq A \leq    [r_{1}, r_{2}, r_{3}]f            $$
Letting $a_{j} \nearrow \tilde{\tilde{a}}$ and $r_{j} \searrow \tilde{\tilde{r}}$, we get (if exists)
$$  f''_{-}(\tilde{\tilde{a}}) \leq A \leq f''_{+}(\tilde{\tilde{r}})   $$
Therefore, if assumptions (a) or (b) holds, then
$$\frac{A}{2}\vast[ \sum_{i = 1}^{n}\alpha_{i}(r_{i})^{2} + \sum_{j = 1}^{m}\beta_{j}(s_{j})^{2} - \sum_{k = 1}^{l}\gamma_{k}(t_{k})^{2} - \vast(\sum_{i = 1}^{n}\alpha_{i}r_{i} + \sum_{j = 1}^{m}\beta_{j}s_{j} - \sum_{k = 1}^{l}\gamma_{k}t_{k}\vast )^{2} $$
$$- \vast \{ \sum_{i = 1}^{n}\alpha_{i}(a_{i})^{2} + \sum_{j = 1}^{m}\beta_{j}(b_{j})^{2} - \sum_{k = 1}^{l}\gamma_{k}(c_{k})^{2} - \vast(\sum_{i = 1}^{n}\alpha_{i}a_{i} + \sum_{j = 1}^{m}\beta_{j}b_{j} - \sum_{k = 1}^{l}\gamma_{k}c_{k}\vast )^{2} \vast\}\vast]$$
is positive and we conclude the result. If the assumption (c) holds, the $f''_{-}$ is left continuous, $f''_{+}$ is right continuous, they are both non-decreasing and $f''_{-} \leq f''_{+}$. Therefore, there exists $\tilde{c} \in [\tilde{\tilde{a}}, \tilde{\tilde{r}}]$ such that $f \in K_{1}^{c}(I)$ with associated constant $\tilde{A} =0$ and again, we can deduce the result.
\end{proof}
\begin{remark}\label{MR2}
Again from the proof of Theorem \ref{MT2}, we obtain the inequalities (\ref{ME5}) and (\ref{ME6}). Now, under assumption (a), (b) or (c) of Theorem \ref{MT2}, $A$ is positive or negative or zero respectively due to argument discussed in the proof. Therefore, we get a better improvement of (\ref{ME3}) then (\ref{ME7}). in this case as follow
$$\sum_{i = 1}^{n}\alpha_{i}f(a_{i}) + \sum_{j = 1}^{m}\beta_{j}f(b_{j}) - \sum_{k = 1}^{l}\gamma_{k}f(c_{k}) - f\vast(\sum_{i = 1}^{n}\alpha_{i}a_{i} + \sum_{j = 1}^{m}\beta_{j}b_{j} - \sum_{k = 1}^{l}\gamma_{k}c_{k}\vast )$$
$$ \leq \frac{A}{2}\vast \{ \sum_{i = 1}^{n}\alpha_{i}(a_{i})^{2} + \sum_{j = 1}^{m}\beta_{j}(b_{j})^{2} - \sum_{k = 1}^{l}\gamma_{k}(c_{k})^{2} - \vast(\sum_{i = 1}^{n}\alpha_{i}a_{i} + \sum_{j = 1}^{m}\beta_{j}b_{j} - \sum_{k = 1}^{l}\gamma_{k}c_{k}\vast )^{2} \vast\}  $$
$$\leq \frac{A}{2}\vast \{ \sum_{i = 1}^{n}\alpha_{i}(r_{i})^{2} + \sum_{j = 1}^{m}\beta_{j}(s_{j})^{2} - \sum_{k = 1}^{l}\gamma_{k}(t_{k})^{2} - \vast(\sum_{i = 1}^{n}\alpha_{i}r_{i} + \sum_{j = 1}^{m}\beta_{j}s_{j} - \sum_{k = 1}^{l}\gamma_{k}t_{k}\vast )^{2} \vast\} $$
\begin{equation}\label{ME9}
\leq \sum_{i = 1}^{n}\alpha_{i}f(r_{i}) + \sum_{j = 1}^{m}\beta_{j}f(s_{j}) - \sum_{k = 1}^{l}\gamma_{k}f(t_{k}) - f\vast(\sum_{i = 1}^{n}\alpha_{i}r_{i} + \sum_{j = 1}^{m}\beta_{j}s_{j} - \sum_{k = 1}^{l}\gamma_{k}t_{k}\vast )
\end{equation}
\end{remark}
Under the assumption of Theorem \ref{MT1} with $f \in K^{c}_{2}(I)$, the reverse of inequality (\ref{ME3}) holds. Now, we give only the statement of Theorem with weaker condition under which the reverse of inequality (\ref{ME3}) also holds for $f \in K^{c}_{2}(I)$.
\begin{theorem}\label{MT3}
Let $\alpha_{i}, \beta_{j}, \gamma_{k} \geq 0$ and $\lambda_{i}, \mu_{j}, \nu_{k} \geq 0$  be coefficients such that their sum $ \alpha = \sum_{i = 1}^{n}\alpha_{i}$, $\beta = \sum_{j = 1}^{m}\beta_{j}$, $\gamma = \sum_{k = 1}^{l}\gamma_{k}$ satisfy $ \alpha + \beta - \gamma = 1$ and $\alpha, \beta \in (0, 1]$; $ \lambda = \sum_{i = 1}^{n}\lambda_{i}$, $\mu = \sum_{j = 1}^{m}\mu_{j}$, $\nu = \sum_{k = 1}^{l}\nu_{k}$ satisfy $ \lambda + \mu - \nu = 1$ and $\lambda, \mu \in (0, 1].$ Let $a_{i}, b_{j}, c_{k} \in [a, c]$ be points such that $c_{k} \in conv\{a_{i}, b_{j}\}$ and $r_{i}, s_{j}, t_{k} \in [c, b]$ be points such that $t_{k} \in conv\{r_{i}, s_{j}\}$ , where
$$   a = \frac{1}{\alpha} \sum_{i = 1}^{n}\alpha_{i}a_{i}\;,\; b = \frac{1}{\beta} \sum_{j = 1}^{m}\beta_{j}b_{j}\;,\;  r = \frac{1}{\alpha} \sum_{i = 1}^{n}\lambda_{i}r_{i}\;,\; s = \frac{1}{\mu} \sum_{j = 1}^{m}\mu_{j}s_{j} . $$ such that
\begin{equation}\label{ME10}
 \tilde{\tilde{a}} \leq \tilde{\tilde{r}}
 \end{equation}
 and $f \in K^{c}_{2}(I)$ for some $c \in [\tilde{\tilde{a}}, \tilde{\tilde{r}}]$. Then if\\
    (a)\\
    $$f''_{-}(\tilde{\tilde{a}}) \leq 0$$
    and
    $$ \sum_{i = 1}^{n}\alpha_{i}(a_{i})^{2} + \sum_{j = 1}^{m}\beta_{j}(b_{j})^{2} - \sum_{k = 1}^{l}\gamma_{k}(c_{k})^{2} - \vast(\sum_{i = 1}^{n}\alpha_{i}a_{i} + \sum_{j = 1}^{m}\beta_{j}b_{j} - \sum_{k = 1}^{l}\gamma_{k}c_{k}\vast )^{2}$$
    $$
    \geq \sum_{i = 1}^{n}\lambda_{i}(r_{i})^{2} + \sum_{j = 1}^{m}\mu_{j}(s_{j})^{2} - \sum_{k = 1}^{l}\nu_{k}(t_{k})^{2} - \vast(\sum_{i = 1}^{n}\lambda_{i}r_{i} + \sum_{j = 1}^{m}\mu_{j}s_{j} - \sum_{k = 1}^{l}\nu_{k}t_{k}\vast )^{2}
    $$
    or\\
    (b)\\
    $$f''_{+}(\tilde{\tilde{r}}) \geq 0$$
    and
    $$ \sum_{i = 1}^{n}\alpha_{i}(a_{i})^{2} + \sum_{j = 1}^{m}\beta_{j}(b_{j})^{2} - \sum_{k = 1}^{l}\gamma_{k}(c_{k})^{2} - \vast(\sum_{i = 1}^{n}\alpha_{i}a_{i} + \sum_{j = 1}^{m}\beta_{j}b_{j} - \sum_{k = 1}^{l}\gamma_{k}c_{k}\vast )^{2}$$
    $$
    \leq \sum_{i = 1}^{n}\lambda_{i}(r_{i})^{2} + \sum_{j = 1}^{m}\mu_{j}(s_{j})^{2} - \sum_{k = 1}^{l}\nu_{k}(t_{k})^{2} - \vast(\sum_{i = 1}^{n}\lambda_{i}r_{i} + \sum_{j = 1}^{m}\mu_{j}s_{j} - \sum_{k = 1}^{l}\nu_{k}t_{k}\vast )^{2}
    $$
    or\\
    (c)
    $$  f''_{-}(\tilde{\tilde{a}}) < 0 < f''_{+}(\tilde{\tilde{r}}) \; \;and \;\; f\; is \; 3-concave,$$
    then reverse of (\ref{ME3}) holds.
\end{theorem}
\begin{remark}\label{MR3}
From the proof of the Theorem \ref{MT3}, we obtain the reverse of inequalities (\ref{ME5}) and (\ref{ME6}). Now, due to the convexity of $F$ on $[a, c]$ and concavity of $F$ on $[c, b]$, for every distinct points $a_{j} \in [a, \tilde{\tilde{a}}]$ and $r_{j} \in [\tilde{\tilde{r}}, b]$, $j = 1, 2, 3,$ we have
$$     [a_{1}, a_{2}, a_{3}]f \geq A \geq    [r_{1}, r_{2}, r_{3}]f            $$
Letting $a_{j} \nearrow \tilde{\tilde{a}}$ and $r_{j} \searrow \tilde{\tilde{r}}$, we get (if exists)
$$  f''_{-}(\tilde{\tilde{a}}) \geq A \geq f''_{+}(\tilde{\tilde{r}})   $$
Now, under assumption (a), (b) or (c) of Theorem \ref{MT2}, $A$ is negative or positive or zero respectively due to argument discussed above. Therefore, we get a better improvement in this case as follow
$$\sum_{i = 1}^{n}\alpha_{i}f(a_{i}) + \sum_{j = 1}^{m}\beta_{j}f(b_{j}) - \sum_{k = 1}^{l}\gamma_{k}f(c_{k}) - f\vast(\sum_{i = 1}^{n}\alpha_{i}a_{i} + \sum_{j = 1}^{m}\beta_{j}b_{j} - \sum_{k = 1}^{l}\gamma_{k}c_{k}\vast )$$
$$ \geq \frac{A}{2}\vast \{ \sum_{i = 1}^{n}\alpha_{i}(a_{i})^{2} + \sum_{j = 1}^{m}\beta_{j}(b_{j})^{2} - \sum_{k = 1}^{l}\gamma_{k}(c_{k})^{2} - \vast(\sum_{i = 1}^{n}\alpha_{i}a_{i} + \sum_{j = 1}^{m}\beta_{j}b_{j} - \sum_{k = 1}^{l}\gamma_{k}c_{k}\vast )^{2} \vast\}  $$
$$\geq \frac{A}{2}\vast \{ \sum_{i = 1}^{n}\alpha_{i}(r_{i})^{2} + \sum_{j = 1}^{m}\beta_{j}(s_{j})^{2} - \sum_{k = 1}^{l}\gamma_{k}(t_{k})^{2} - \vast(\sum_{i = 1}^{n}\alpha_{i}r_{i} + \sum_{j = 1}^{m}\beta_{j}s_{j} - \sum_{k = 1}^{l}\gamma_{k}t_{k}\vast )^{2} \vast\} $$
\begin{equation}\label{ME11}
\geq \sum_{i = 1}^{n}\alpha_{i}f(r_{i}) + \sum_{j = 1}^{m}\beta_{j}f(s_{j}) - \sum_{k = 1}^{l}\gamma_{k}f(t_{k}) - f\vast(\sum_{i = 1}^{n}\alpha_{i}r_{i} + \sum_{j = 1}^{m}\beta_{j}s_{j} - \sum_{k = 1}^{l}\gamma_{k}t_{k}\vast )
\end{equation}
\end{remark}
\begin{theorem}\label{MT4}
Let $\mathcal{I} \subset \mathbb{R}$ be a closed interval, let $[a, b] \subset \mathcal{I}$, let function $g_{i} \in \mathbb{X}_{[a, b]}$ and function $h_{i} \in \mathbb{X}_{\mathcal{I} \backslash (a, b)}$ for $i = 1,2$. Let $f \in K^{c}_{1}(\mathcal{I})$ be continuous function such that $f(g_{i}), f(h_{i}) \in \mathbb{X}$. If a pair of unital positive linear functionals $L, H :\mathbb{X} \rightarrow \mathbb{R}$ satisfies
\begin{equation}\label{ME12}
L(g_{i}) = H(h_{i}) \;\;and\;\; H(h_{1}^{2}) - L(g_{1}^{2}) = H(h_{2}^{2}) - L(g_{2}^{2}),\;\;i \;=\;1,2,
\end{equation}
then inequality
\begin{equation}\label{ME13}
H(f(h_{1})) - L(f(g_{1})) \leq H(f(h_{2})) - L(f(g_{2}))
\end{equation}
holds.
\end{theorem}
\begin{proof}
Since $f \in K^{c}_{1}(\mathcal{I})$, there exists a constant $A$ such that $F(x) = f(x) - \frac{A}{2}x^{2}$ is concave on $\mathcal{I} \cap (-\infty, c]$, therefore by reverse of (\ref{IE5}) for $F$ on $\mathcal{I} \cap (-\infty, c]$, we get
\begin{eqnarray*}
0 &\geq& H(F(h_{1})) - L(F(g_{1}))\\
&=& H(f(h_{1})) - L(f(g_{1})) - \frac{A}{2}( H(h_{1}^{2}) - L(g_{1}^{2}) )
\end{eqnarray*}
Also, since $F(x) = f(x) - \frac{A}{2}x^{2}$ is convex on $\mathcal{I} \cap [c, \infty)$, therefore by (\ref{IE5}) for $F$ on $\mathcal{I} \cap (-\infty, c]$, we get
\begin{eqnarray*}
0 &\leq& H(F(h_{2})) - L(F(g_{2}))\\
&=& H(f(h_{2})) - L(f(g_{2})) - \frac{A}{2}( H(h_{2}^{2}) - L(g_{2}^{2}) )
\end{eqnarray*}
From above, we have
\begin{multline*}
H(f(h_{1})) - L(f(g_{1})) - \frac{A}{2}( H(h_{1}^{2}) - L(g_{1}^{2}) )\\ \leq 0 \leq\\ H(f(h_{2})) - L(f(g_{2})) - \frac{A}{2}( H(h_{2}^{2}) - L(g_{2}^{2}) ).
\end{multline*}
So
\begin{multline*}
H(f(h_{1})) - L(f(g_{1})) - \frac{A}{2}( H(h_{1}^{2}) - L(g_{1}^{2}) ) \\\leq H(f(h_{2})) - L(f(g_{2})) - \frac{A}{2}( H(h_{2}^{2}) - L(g_{2}^{2}) ),
\end{multline*}
therefore, by the use of (\ref{ME12}), we get (\ref{ME13}).
\end{proof}
\begin{remark}\label{MR4}
From the proof of the Theorem \ref{MT4}, we have
\begin{equation}\label{ME14}
H(f(h_{1})) - L(f(g_{1}))\leq  \frac{A}{2}( H(h_{1}^{2}) - L(g_{1}^{2}) )
\end{equation}
and
\begin{equation}\label{ME15}
H(f(h_{2})) - L(f(g_{2}))\geq  \frac{A}{2}( H(h_{2}^{2}) - L(g_{2}^{2}) )
\end{equation}
So, under assumption (\ref{ME12}), we can get a better improvement of (\ref{ME13}) as follow
\begin{multline}\label{ME16}
H(f(h_{1})) - L(f(g_{1}))\leq  \\ \frac{A}{2}( H(h_{1}^{2}) - L(g_{1}^{2}) ) \vast( = \frac{A}{2}( H(h_{2}^{2}) - L(g_{2}^{2}) )  \vast) \\ \leq H(f(h_{2})) - L(f(g_{2}))
\end{multline}
\end{remark}
\begin{corollary}\label{MC1}
Let $\mathcal{I} \subset \mathbb{R}$ be a closed interval, let $[a, b] \subset \mathcal{I}$, let function $g_{i} \in \mathbb{X}_{[a, b]}$ for $i = 1,2$. Let $f \in K^{c}_{1}(\mathcal{I})$ be continuous function such that $f(g_{i}) \in \mathbb{X}$. If a unital positive linear functionals $L :\mathbb{X} \rightarrow \mathbb{R}$ satisfies implication
(\ref{ME12}) $\Rightarrow$ (\ref{ME13}) for $L = H$ such that
\begin{equation}\label{ME17}
L(g_{1}^{2}) - (L(g_{1}))^{2} = L(g_{2}^{2}) - (L(g_{2}))^{2}
\end{equation}
then following inequality holds
\begin{equation}\label{ME18}
L(f(g_{1})) - f(L(g_{1})) \leq L(f(g_{2})) - f(L(g_{2}))
\end{equation}
\end{corollary}
\begin{corollary}\label{MC2}
Let $ [a_{1}, b_{1}] \subseteq...\subseteq[a_{n - 1}, b_{n - 1}] \subseteq \mathcal{I}.$ Let function $g_{1}, h_{1} \in \mathbb{X}_{[a_{1}, b_{1}]},$ let $g_{k}, h_{k} \in \mathbb{X}_{[a_{k}, b_{k}] \backslash (a_{k - 1}, b_{k - 1})}$ for $k = 2,...,n - 1,$ and let function $g_{n}, h_{n} \in \mathbb{X}_{\mathcal{I} \backslash (a_{n - 1}, b_{n - 1})}.$ Let $f \in K^{c}_{1}(\mathcal{I})$ be continuous function such that $f(g_{i}) \in \mathbb{X}.$\\
If an n-tuple of unital positive linear functionals $L_{i}:\mathbb{X} \rightarrow \mathbb{R}$ satisfies
\begin{equation}\label{ME19}
L_{i}(g_{i}) = L_{i + 1}(g_{i + 1})\;\;and\;\;L_{i}(h_{i}) = L_{i + 1}(h_{i + 1})\;\;for\;i\;=1,...,n -1,
\end{equation}
such that
\begin{equation}\label{ME20}
L_{i + 1}(g_{i + 1}^{2}) - L_{i}(g_{i}^{2}) = L_{i + 1}(h_{i + 1}^{2}) - L_{i}(h_{i}^{2}),
\end{equation}
then
\begin{equation}\label{ME21}
L_{i + 1}f(g_{i + 1}) - L_{i}f(g_{i}) \leq L_{i + 1}f(h_{i + 1}) - L_{i}f(h_{i})\;\;for\;i\;=1,...,n -1.
\end{equation}
\end{corollary}
\begin{corollary}\label{MC3}
Let $\mathcal{I} \subset \mathbb{R}$ be a closed interval, and let functions $g_{i}, h_{i} \in \mathbb{X}_{\mathcal{I}}$ for $i = 1,...,n$. Let $f \in K^{c}_{1}(\mathcal{I})$ be continuous function such that $f(g_{i}) , f(h_{i}) \in \mathbb{X}.$\\
Then every n-tuple of positive linear functionals $L_{i}:\mathbb{X} \rightarrow \mathbb{R}$ with $\sum_{i = 1}^{n} L_{i}(1) =1$ such that
\begin{equation}\label{ME22}
\sum_{1 = 1}^{n} L_{i}((g_{i})^{2}) - \big(\sum_{1 = 1}^{n} L_{i}(g_{i})\big)^{2} = \sum_{1 = 1}^{n} L_{i}((h_{i})^{2}) - \big(\sum_{1 = 1}^{n} L_{i}(h_{i})\big)^{2}
\end{equation}
 satisfies the inclusion
\begin{equation}\label{ME23}
\sum_{1 = 1}^{n} L_{i}(g_{i}), \sum_{1 = 1}^{n} L_{i}(g_{i}) \in \mathcal{I}
\end{equation}
and the inequality
\begin{equation}\label{ME24}
\sum_{1 = 1}^{n} L_{i}(f(g_{i})) - f\big(\sum_{1 = 1}^{n} L_{i}(g_{i})\big) \leq \sum_{1 = 1}^{n} L_{i}(f(h_{i})) - f\big(\sum_{1 = 1}^{n} L_{i}(h_{i})\big)
\end{equation}
\end{corollary}
\begin{theorem}\label{MT5}
Let $\mathcal{I} \subset \mathbb{R}$ be a closed interval, let $[a, b] \subset \mathcal{I}$, let function $g_{i} , g^{*}_{i}\in \mathbb{X}_{[a, b]}$ for $i = 1,...,n$ and $h_{i} , h^{*}_{i}\in \mathbb{X}_{\mathcal{I}\backslash(a, b)}$ for $j = 1,...,m.$ Let $f \in K^{c}_{1}(\mathcal{I})$ be continuous function such that $f(g_{i}), f(g^{*}_{i}), f(h_{i}), f(h^{*}_{i}) \in \mathbb{X}.$\\
If two pair of n-tuple of positive linear functionals $L_{i}, L^{*}_{i}, H_{j}, H^{*}_{j} :\mathbb{X} \rightarrow \mathbb{R}$ with $$\sum_{i = 1}^{n} L_{i}(1)
= \sum_{1 = 1}^{n} L^{*}_{i}(1) = \sum_{j = 1}^{m} H_{j}(1) = \sum_{j = 1}^{m} H^{*}_{j}(1) = 1$$ satisfy
\begin{equation}\label{ME25}
\sum_{j = 1}^{m} H_{j}(h_{j}) = \sum_{i = 1}^{n} L_{i}(g_{i})\;\;and\;\; \sum_{j = 1}^{m} H^{*}_{j}(h^{*}_{j}) = \sum_{i = 1}^{n} L^{*}_{i}(g^{*}_{i})
\end{equation}
and
\begin{equation}\label{ME26}
\sum_{j = 1}^{m} H_{j}((h_{j})^{2}) - \sum_{i = 1}^{n} L_{i}((g_{i})^{2}) = \sum_{j = 1}^{m} H^{*}_{j}((h^{*}_{j})^{2}) - \sum_{i = 1}^{n} L^{*}_{i}((g^{*}_{i})^{2}).
\end{equation}
Then
\begin{equation}\label{ME27}
\sum_{j = 1}^{m} H_{j}f(h_{j}) - \sum_{i = 1}^{n} L_{i}f(g_{i}) \leq \sum_{j = 1}^{m} H^{*}_{j}f(h^{*}_{j}) - \sum_{i = 1}^{n} L^{*}_{i}f(g^{*}_{i}s)
\end{equation}
\end{theorem}

\section*{Acknowledgement} % A simple way to write an acknowledgement
The authors express their gratitude to the referees for their valuable comments.

%\bibliography{mmnsample}

%\bibliographystyle{mmn}

\end{document}